\documentclass[11pt]{article}
\usepackage{amscd}
\usepackage{}
\usepackage{mathrsfs}
\usepackage{amsfonts}
\usepackage{amsfonts,amssymb,mathrsfs}
\usepackage{amsmath,amscd}
\usepackage[dvips]{graphicx}
\usepackage[all]{xy}
\usepackage{graphicx}
\usepackage{fancyhdr}
\usepackage{mathptm,pslatex}
\usepackage{amsthm}

\begin{document}
\sloppy
\newcommand{\dickebox}{{\vrule height5pt width5pt depth0pt}}
\newtheorem{thm}{Theorem}[section]
\newtheorem{defn}[thm]{Definition}
\newtheorem{lem}[thm]{Lemma}
\newtheorem{prop}[thm]{Proposition}
\newtheorem{cor}[thm]{Corollary}
\newtheorem{rem}[thm]{Remark}
\newtheorem{exa}[thm]{Example}

\newcommand{\lra}{\longrightarrow}
\newcommand{\xra}{\xrightarrow}
\newcommand{\ra}{\rightarrow}
\newcommand{\ol}{\overline}
\newcommand{\ul}{\underline}
\newcommand{\rsd}{{\rm resol.dim\,}}
\newcommand{\pd}{{\rm proj.dim\, }}
\newcommand{\id}{{\rm inj.dim\, }}
\newcommand{\gd}{{\rm gl.dim\, }}
\newcommand{\fd} {{\rm fin.dim\, }}
\newcommand{\repd} {{\rm rep.dim\, }}
\newcommand{\rad}{{\rm rad\,}}
\newcommand{\soc}{{\rm soc\,}}
\renewcommand{\top}{{\rm top\,}}
\newcommand{\ind}{{\rm ind\,}}
\newcommand{\add}{{\rm add\, }}
\newcommand{\Add}{{\rm Add\, }}
\newcommand{\Fac}{{\rm Fac\,}}
\newcommand{\Hom}{{\rm Hom \, }}
\newcommand{\End}{{\rm End\, }}
\newcommand{\rHom}{{\rm\bf R}{\rm Hom}\,}
\newcommand{\StHom}{{\rm \underline{Hom} \, }}
\newcommand{\Img}{{\rm Im}\,\,}
\newcommand{\Coker}{{\rm Coker}\,\,}
\newcommand{\Ker}{{\rm Ker}\,\,}
\newcommand{\Ext}{{\rm Ext}}
\newcommand{\Yext}{{\rm Yext}}
\newcommand{\E}{{\rm E}}
\newcommand{\dk}{{\rm dim_{_{K}}}}
\newcommand{\cpx}[1]{#1^{\bullet}}
\newcommand{\ccpx}[1]{#1_{\bullet}}
\newcommand{\D}[1]{{\mathscr D}(#1)}
\newcommand{\Dz}[1]{{\rm D}^+(#1)}
\newcommand{\Df}[1]{{\rm D}^-(#1)}
\newcommand{\Db}[1]{{\mathscr D}^b(#1)}
\newcommand{\C}[1]{{\mathscr C}(#1)}
\newcommand{\Cz}[1]{{\rm C}^+(#1)}
\newcommand{\Cf}[1]{{\rm C}^-(#1)}
\newcommand{\Cb}[1]{{\mathscr C}^b(#1)}
\newcommand{\K}[1]{{\mathscr K}(#1)}
\newcommand{\Kz}[1]{{\rm K}^+(#1)}
\newcommand{\Kf}[1]{{\mathscr K}^-(#1)}
\newcommand{\Kb}[1]{{\mathscr K}^b(#1)}
\newcommand{\Modc}{\ensuremath{\mbox{{\rm -Mod}}}}
\newcommand{\modc}{\ensuremath{\mbox{{\rm -mod}}}}
\newcommand{\stmodc}[1]{#1\mbox{{\rm -{\underline{mod}}}}}
\newcommand{\Proj}[1]{#1\mbox{{\rm -Proj}}}
\newcommand{\proj}[1]{#1\mbox{{\rm -proj}}}
\newcommand{\inj}[1]{#1\mbox{{\rm -inj}}}
\newcommand{\op}{^{\rm op}}
\newcommand{\otimesL}{\otimes^{\rm\bf L}}
\newcommand{\gm}{{\rm _{\Gamma_M}}}
\newcommand{\gmr}{{\rm _{\Gamma_M^R}}}
\def\vez{\varepsilon}\def\bz{\bigoplus}\def\sz{\oplus}

\newcommand{\scr}[1]{\mathscr #1}
\newcommand{\al}[1]{\mathcal #1}

{\Large \bf
\begin{center}  Transfer of derived equivalences from subalgebras to endomorphism algebras II
 \end{center}}
\medskip

\centerline{\sc Shengyong Pan$^*$  and Jiahui Yu}

\begin{center} School of Mathematics and Statistics,\\
 Beijing Jiaotong University,  Beijing 100044,\\
People's Republic of China\\
\end{center}

\renewcommand{\thefootnote}{\alph{footnote}}
\setcounter{footnote}{-1} \footnote{ $^*$ Corresponding author. Email: shypan@bjtu.edu.cn, 22121599@bjtu.edu.cn}

\renewcommand{\thefootnote}{\alph{footnote}}
\setcounter{footnote}{-1} 
\renewcommand{\thefootnote}{\alph{footnote}}
\setcounter{footnote}{-1} 

\medskip
\begin{abstract} We investigate derived equivalences between subalgebras of some $\Phi$-Auslander-Yoneda algebras from a class of $n$-angles
in weakly $n$-angulated categories.
The derived equivalences are obtained by transferring subalgebras induced by $n$-angles to endomorphism algebras induced by approximation sequences. Then we extend
our constructions \cite{BP} to $n$-angle cases. Finally, we give an explicit example to illustrate our result.
\end{abstract}

\section{Introduction}

In the representation theory of algebras, derived equivalences have been
shown to preserve many homological invariants and provide new bridges between algebras and geometry. Rickard's Morita theory for
derived categories states that two rings are derived equivalent if and only if there is
a tilting complex over one ring with endomorphism ring isomorphic to the other ring \cite{Ri1, Ri2}.
However, in general, it is very difficult to describe the derived equivalence class of a given ring.
One idea is to study derived equivalent rings can be related by a sequence of such elementary derived equivalences. Many such kind of derived equivalent rings comes from mutations sequences of objects in categories, where approximations play a central role. This occurs in many aspects in algebra and geometry, such as mutations of tilting modules \cite{Happel1998b}, mutations of cluster tilting objects \cite{Buan2006}, mutations of silting objects \cite{Aihara2012}, mutations of exceptional sequences \cite{Rudakov1990}.

It is interesting to know whether the mutation sequences always give derived equivalent endomorphism rings.
In \cite{HX1}, Hu and Xi got
derived equivalences by $\mathcal {D}$-split sequences via tilting modules. Indeed, they used some special approximation sequences to get
derived equivalences. In \cite{HX2}, Hu and Xi constructed derived equivalences between $\Phi$-Auslander-Yoneda algebras
from given $\nu$-stable derived equivalences. Furthermore, Hu, Koenig and Xi \cite{HXK} proved that certain
triangles with symmetric approximations and some conditions give rise to derived equivalences of quotient algebras of endomorphism algebras of objects in the sequences modulo some particularly defined ideals.
After that Chen \cite{Chen2013} generalized their result to $n$-angles in $n$-angulated categories.
Chen and Hu \cite{CH} introduced
symmetric approximation sequences in additive categories and weakly $n$-angulated categories which include (higher) Auslander-Reiten sequences (triangles) and showed that such sequences always give rise to derived equivalences between the quotient rings of endomorphism rings of objects in the sequences modulo some ghost and coghost ideals. Recently, Pan \cite{Pan} proved that symmetric approximation sequences in $n$-exangulated categories give rise to derived equivalences between quotient algebras of locally $\Phi$-Beilinson-Green algebras in the principal diagonals modulo some factorizable ghost and coghost ideals by the locally finite tilting family.

A natural question arises, how about derived equivalences between subalgebras?

In \cite{Chen2014}, Chen constructed derived equivalences
between some subalgebras related to $\Phi$-Auslander-Yoneda algebras induced by exact sequences in abelian categories.
Thomas and Pan investigate derived equivalences between subalgebras of some $\Phi$-Auslander-Yoneda algebras from a class of triangles
in triangulated categories. The derived equivalences are obtained by transferring subalgebras induced by triangles to endomorphism algebras induced by approximation sequences

We extend in this paper Thomas-Pan's main result \cite[Theorem 1.1]{BP} to a more general case and construct
derived equivalences between subalgebras of $\Phi$-Auslander-Yoneda algebras induced from $n$-angles in a weakly $n$-angulated category $(\al T,\Sigma)$ ($n\geq 2$). In order to state our result precisely,
we first fix some notation. Let $\Phi$ be an admissible set of $\mathbb{N}$ and $X$ an object in $\al T$. The $\Phi$-Auslander-Yoneda algebra
of $X$ is $\E_{\al T}^\Phi(X):=\bigoplus_{i\in \Phi} \Hom_{\mathcal T}(X, X[i])$.  The matrix rings
$\begin{pmatrix}
\widetilde{\E_{\al T}^\Phi(X)}&     \widetilde{\E_{\al T}^\Phi(X,M)}  \\
\E_{\al T}^\Phi(M,X)    &\E_{\al T}^\Phi(M)
\end{pmatrix}$
and
$\begin{pmatrix}
\widetilde{\E_{\al T}^\Phi(Y)}&     \E_{\al T}^\Phi(Y,M)\\
\widetilde{\E_{\al T}^\Phi(M,Y)}    &\E_{\al T}^\Phi(M)
\end{pmatrix}$
are some subrings of $\E_{\al T}^\Phi(X\oplus M)$ and $\E_{\al T}^\Phi(Y\oplus M)$, respectively. For definitions of these rings, we refer the reader to Section 2.

Our main result is as follows:

\begin{thm}\label{1.1}
Let $(\al T,\Sigma)$ be a weakly $n$-angulated category ($n\geq 2$),
and let $M$ be an object in $\mathcal {T}$. Let
 $$
 X\xra f M_1\xra{f_1} M_2\ra\cdots\xra{f_{n-1}} M_{n-2}\xra g Y\xra h \Sigma X
 $$
be an $n$--angle in $\mathcal {T}$ with $M_j\in\add(M)$ for all $j=1,\cdots,n-2$, assume that $\Hom_{\al T}(M,\Sigma^i(X))=0=\Hom_{\al T}(Y,\Sigma^i(M))$ for $0\neq i\in\Phi$,
and $\Hom_{\al T}(\Sigma(X),Y)=\Hom_{\al T}(\Sigma(X),M)=\Hom_{\al T}(\Sigma(M),Y)=0$.
Then the rings
$\begin{pmatrix}
\widetilde{\E_{\al T}^\Phi(X)}&     \widetilde{\E_{\al T}^\Phi(X,M)}  \\
\E_{\al T}^\Phi(M,X)    &\E_{\al T}^\Phi(M)
\end{pmatrix}$
and
$\begin{pmatrix}
\widetilde{\E_{\al T}^\Phi(Y)}&     \E_{\al T}^\Phi(Y,M)\\
\widetilde{\E_{\al T}^\Phi(M,Y)}    &\E_{\al T}^\Phi(M)
\end{pmatrix}$
are derived equivalent.
\end{thm}

\begin{rem}
(1) By transferring subalgebras induced by $n$-angles to endomorphism algebras induced by approximation sequences,
 we give a simple approach to prove our result. For more details, we refer to Lemma \ref{3.4}.
 We also give an alternative proof to calculate the endomorphism rings of tilting complexes, we refer to  Lemma \ref{2.5}.
 Although the method is similar to that in \cite[Theorem 1.1]{Pan}, but here we need to check more details since the conditions in this theorem is more complicated.

 (2) We focus on derived equivalences between subalgebras which is different from derived equivalences between quotient algebras by Chen \cite{Chen2013}.

(3) If $n=3$, then we re-obtain the main result  \cite[Theorem 1.1]{BP}.
\end{rem}

This paper is organized as follows. In Section 2, we review some
basic facts on derived equivalences and $\Phi$-Auslander-Yoneda algebras and give
some notions. Then we prove Theorem \ref{1.1}. Finally, we give an explicit example to illustrate the main theorem.

\section{Proof of Theorem 1.1}
In this section, we shall review basic definitions and facts which will be useful in the proofs later on.
In particular,  we recall the definition of
$\Phi$-Auslander-Yoneda algebras, where $\Phi$ is an
admissible set of $\mathbb{N}$, and we construct derived
equivalences between subalgebras of some $\Phi$-Auslander-Yoneda algebras from given $n$-angles in a weakly $n$-angulated category ($n\geq 2$).

\subsection{Conventions}
We begin by briefly recalling some definitions and notations on
derived categories and derived equivalences.

Let $\mathscr{A}$ be an additive category. For two morphisms $\alpha:
X\ra Y$ and $\beta: Y\ra Z$, their composition is denoted by
$\alpha\beta$, where $X$ and $Y$ are in $\mathscr{A}$. A
complex $\cpx{X}=(X^i,d_{X}^i)$ over $\mathscr{A}$ is a sequence of
objects $X^i$ and morphisms $d_{X}^i$ in $\mathscr{A}$ of the form:
$$\cdots \ra X^i\stackrel{d^i}\ra X^{i+1}\stackrel{d^{i+1}}\ra
X^{i+1}\ra\cdots,$$ such that $d^id^{i+1}=0$ for all
$i\in\mathbb{Z}$. If $\cpx{X}=(X^i,d_{X}^i)$ and
$\cpx{Y}=(Y^i,d_{Y}^i)$ are two complexes, then a morphism $\cpx{f}:
\cpx{X}\ra\cpx{Y}$ is a sequence of morphisms $f^i: X^i\ra Y^i$ of
$\mathscr{A}$ such that $d^i_{X}f^{i+1}=f^id^i_{Y}$ for all
$i\in\mathbb{Z}$. The morphism $\cpx{f}$ is called a chain map between
$\cpx{X}$ and $\cpx{Y}$. The category of
complexes over $\mathscr{A}$ with chain maps is denoted by
$\C{\mathscr{A}}$.
The homotopy category of complexes over $\mathscr{A}$ is denoted by
$\K{\mathscr{A}}$. If $\mathscr{A}$ is an abelian category, then we denote by $\D{\mathscr{A}}$
the derived category of complexes over $\mathscr{A}$. It is
well known that, for an abelian category $\mathscr{A}$, the
categories $\K{\mathscr{A}}$ and $\D{\mathscr{A}}$ are triangulated
categories. For basic results on triangulated categories, we refer
the reader to \cite{N}.

Let $R$ be a commutative artinian ring and let $A$ be an Artin $R$-algebra.
Denote by $A$-Mod and $A$-mod the category of left $A$-modules and finitely generated left
$A$-modules, respectively. The full subcategory of $A$-Mod and $A$-mod consisting of
projective modules is denoted by $\Proj{A}$ and $\proj{A}$, respectively.
 Let $\Kb{A\Modc}$ denote the
homotopy category of bounded complexes of $A$-modules and let
$\Db{A\Modc}$ denote the bounded derived category of A\Modc, respectively.

The following theorem is a key ingredient of Morita theory on derived equivalences for module categories of
rings or algebras which was established by Rickard \cite{Ri1}.

\begin{thm} \label{R} $\rm \cite[Theorem \;6.4]{Ri1}$
Let $A$ and $B$ be rings. The following conditions are equivalent.

$(i)$ $\Db{A\Modc}$ and $\Db{B\Modc}$ are equivalent
as triangulated categories.

$(ii)$ $\Kf{\Proj{A}}$ and $\Kf{\Proj{B}}$ are equivalent as
triangulated categories.

$(iii)$ $\Kb{\Proj{A}}$ and $\Kb{\Proj{B}}$ are equivalent as
triangulated categories.

$(iv)$ $\Kb{\proj{A}}$ and $\Kb{\proj{B}}$ are equivalent as
triangulated categories.

$(v)$ $B$ is isomorphic to $\End_{\Db{A\Modc}}(\cpx{T})$ for some complex
$\cpx{T}$ in $\Kb{\proj{A}}$ satisfying

         \qquad $(a)$ $\Hom_{\Db{A\Modc}}(\cpx{T},\cpx{T}[n])=0$
         for all $n\neq 0$.

         \qquad $(b)$ $\add(\cpx{T})$, the category of direct summands of
          finite direct sums of copies of $\cpx{T}$, generates
          $\Kb{\proj{A}}$ as a triangulated category.
\end{thm}
\noindent{\bf Remarks.} (1) The rings $A$ and $B$ are said to be derived equivalent if $A$
and $B$ satisfy the conditions of the above theorem.

(2) The complex $\cpx{T}\in \Kb{\proj{A}}$
in Theorem \ref{R} (v) which satisfies the conditions (a) and (b) is
called a {\em tilting complex} for $A$.

\subsection{$\Phi$-Auslander-Yoneda algebras}
Let $\mathbb Z$ be the
set of all integers, and let $\mathbb N$ be the set of natural numbers.
 Recall from \cite{HX2} that an
admissible subset $\Phi$ of $\mathbb Z$ is defined as a subset of $\mathbb Z$  containing $0$ and the following condition is
satisfied:

If $i,j,k\in \Phi$ satisfy that $i+j+k\in \Phi$, then $i+j\in \Phi$ if and
only if $j+k\in \Phi$.

Recall that the $\Phi$-Auslander-Yoneda algebras defined in \cite{HX2} as follows. Let $R$ be a commutative ring, and $\Phi$ an admissible subset of $\mathbb Z$.
For a triangulated $R$-category $\al T$ with shift functor [1]
and a triangulated functor $F$ from $\al T$ to $\al T$. The $\Phi$-orbit category
${\mathcal T}^{\Phi}$ of $\mathcal T$ is a category in which the
objects are the same as that of $\mathcal T$, and the morphism set
between two objects $X$ and $Y$ is defined to be

$$ \Hom_{{\mathcal T}^{\Phi}}(X,Y):=\displaystyle \bigoplus_{i\in \Phi} \Hom_{\mathcal T}(X, Y[i])\in R\Modc,$$
and the composition is defined in the obvious way. Since $\Phi$ is
admissible, ${\mathcal T}^{\Phi}$ is an additive $R$-category. In
particular, $\Hom_{{\mathcal T}^{\Phi}}(X, X)$ is an $R$-algebra
(which may not be  artinian), and $\Hom_{{\mathcal
T}^{\Phi}}(X, Y)$ is an $\Hom_{{\mathcal T}^{\Phi}}(X,
X)$-$\Hom_{{\mathcal T}^{\Phi}}(Y, Y)$-bimodule. The algebra $\Hom_{{\mathcal T}^{\Phi}}(X, X)$ is called the $\Phi$-Auslander-Yoneda algebra
of $X$.

\subsection{Proof of the main result}
Let $\mathcal {D}$ be a full
subcategory of $\al T$. Suppose $X$ is an object of $\al T$. A
morphism $f:X\ra D$ is called a left $\mathcal {D}$-approximation of $X$ if $D\in\mathcal {D}$, and
for any morphism $f': X\ra D'$ with $D'\in\mathcal {D}$, there is a morphism $f'': D\ra D$
such that $f'=ff''$. Similarly, we have the notion of a right $\mathcal {D}$-approximation of $X$.

Let $M\in\al T$. Assume that there is an $n$-angle
$X\xra f M_1\xra{f_1} M_2\ra\cdots\xra{f_{n-3}} M_{n-2}\xra g Y\xra h \Sigma X$ in $\al T$.
Then we have two $n$-angles
\begin{equation}
 M\oplus X\xra{\ol f} M\oplus M_1\xra{\ol f_1} M_2\ra\cdots\xra{f_{n-3}} M_{n-2}\xra g Y\xra h \Sigma X\\
\end{equation}
\begin{equation}
 X\xra{f}  M_1\xra{f_1} M_2\ra\cdots\xra{\ol f_{n-3}} M\oplus M_{n-2}\xra{\ol g} M\oplus Y\xra{h} \Sigma X
\end{equation}
in $\al T$, where $\ol{f}=\begin{pmatrix}
1 &0 \\
0 &f
\end{pmatrix}: M\oplus M_n\ra M\oplus Y$, $\ol{f_1}=\begin{pmatrix}
0  \\
f_1
\end{pmatrix}:M\oplus M_1\ra  M_2$, $\ol{f_{n-1}}=(0,f_{n-3}):M_{n-3}\ra M\oplus M_{n-2}$ and $\ol{g}=\begin{pmatrix}
1 &0 \\
0 &g
\end{pmatrix}: M\oplus M_{n-2}\ra  M\oplus Y$.

Let $W:=M\oplus X \oplus\bigoplus^n_{i=1} M_i\oplus Y$. We thus have a $\Phi$-Auslander-Yoneda algebra of $W$ as follows:
$$\Gamma=
\begin{pmatrix}
 \E^{\Phi}(X,X) &  (X,M)  &  \E^{\Phi}(X,M_1) & \cdots & \E^{\Phi}(X,M_n)& \E^{\Phi}(X,Y)\\
\E^{\Phi}(M,X)      & \E^{\Phi}(M) & \E^{\Phi}(M,M_1) & \cdots & \E^{\Phi}(M,M_n)&\E^{\Phi}(M,Y)\\
\E^{\Phi}(M_1,X)& \E^{\Phi}(M_1,M)    & \E^{\Phi}(M_1,M_1)  & \cdots  & \E^{\Phi}(M_1,M_n)& \E^{\Phi}(M_1,Y)\\
  \vdots & \vdots & \vdots & \vdots& \vdots\\
  \E^{\Phi}(M_{n-1},X) &\E^{\Phi}(M_{n-1},M)    & \E^{\Phi}(M_{n-1},M_1)&\cdots  & \E^{\Phi}(M_{n-1},M_n)& \E^{\Phi}(M_{n-1},Y)\\
  \E^{\Phi}(M_n,X) &\E^{\Phi}(M_n,M) &\E^{\Phi}(M_n,M_1) &\cdots &\End(M_n)&\E^{\Phi}(M_n,Y)\\
\E^{\Phi}(Y,X)      &\E^{\Phi}(Y,M) & \E^{\Phi}(Y,M_1) & \cdots & \E^{\Phi}(Y,M_n)&\End(Y)\\
\end{pmatrix}.
$$
Now we take the subalgebra $\Lambda$ of $\Gamma$ as follows:
$$\Lambda
=
\begin{pmatrix}
\widetilde{\E^{\Phi}(X)}  &  \widetilde{\E^{\Phi}(X,M)}     &  \widetilde{\E^{\Phi}(X,M_1)} & \cdots & \widetilde{\E^{\Phi}(X,M_n)}& \widetilde{\E^{\Phi}(X,Y)}\\
 \E^{\Phi}(M,X)      &\E^{\Phi}(M) & \E^{\Phi}(M,M_1) & \cdots & \E^{\Phi}(M,M_n)&\widetilde{\E^{\Phi}(M,Y)}\\
   \E^{\Phi}(M_1,X)& \E^{\Phi}(M_1,M)    & \E^{\Phi}(M_{n-2},M_1)  & \cdots  & \E^{\Phi}(M_1,M_n)& \widetilde{\E^{\Phi}(M_1,Y)}\\
  \vdots & \vdots & \vdots & \vdots& \vdots\\
  \E^{\Phi}(M_{n-1},X) &\E^{\Phi}(M_{n-1},M)    & \E^{\Phi}(M_{n-1},M_1)&\cdots  & \E^{\Phi}(M_{n-1},M_n)& \widetilde{\E^{\Phi}(M_{n-1},Y)}\\
  \E^{\Phi}(M_n,X) &\E^{\Phi}(M_n,M) &\E^{\Phi}(M_n,M_1) &\cdots &\E^{\Phi}(M_n)&\widetilde{\E^{\Phi}(M_n,Y)}\\
\E^{\Phi}(Y,X)      &\E^{\Phi}(Y,M) & \E^{\Phi}(Y,M_1) & \cdots & \E^{\Phi}(Y,M_n)&\widetilde{\E^{\Phi}(Y)}\\
\end{pmatrix},
$$
where
$$
\begin{array}{rl}
\widetilde{\End(X)}=\{t\in\End(X)\mid tf=ft'\;\text{for some}\; t'\in\End_{\al C}(M_1) \},\\
\widetilde{\Hom(X,M)}=\{t\in\Hom(X,M)\mid t \;\text{factors through}\; f \;\text{in}\; \mathcal {C}\},\\
\widetilde{\Hom(X,M_i)}=\{t\in\Hom(X,M_i)\mid t \;\text{factors through}\; f \;\text{in}\; \mathcal {C}, i=1,\cdots,n\},\\
\widetilde{\Hom(X,Y)}=\{t\in\Hom(X,Y)\mid t \;\text{factors through}\; f\; \text{and}\; g\;\text{in}\; \mathcal {C}\},\\
\widetilde{\Hom(M,Y)}=\{t\in \Hom(M,Y)\mid t \;\text{factors through}\; g \;\text{in}\; \mathcal {C} \},\\
\widetilde{\Hom(M_i,Y)}=\{t_i\in \Hom(M,Y)\mid t_i \;\text{factors through}\; g \;\text{in}\; \mathcal {C},i=1,\cdots,n \},\\
\widetilde{\End(Y)}=\{t\in \End(Y)\mid gt=t'g\;\text{for some} \; t'\in\End_{\al C}(M_n) \},\\
\widetilde{\E^{\Phi}(X)}=\{t=(t_i)_{i\in\Phi}\in\E^{\Phi}(X)\mid t_i\Sigma^i(f) \;\text{factors through}\; f \;\text{in}\; \mathcal {T}\;\text{for}\; i\in\Phi \},\\
\widetilde{\E^{\Phi}(X,M)}=\{t=(t_i)_{i\in\Phi}\in\E^{\Phi}(X,M)\mid t_i \;\text{factors through}\; f \;\text{in}\; \mathcal {T}\;\text{for}\; i\in\Phi \},\\
\widetilde{\E^{\Phi}(X,Y)}=\{t=(t_i)_{i\in\Phi}\in\E^{\Phi}(X,Y)\mid t_i \;\text{factors through}\; f\; \text{and}\; \Sigma^i(g) \;\text{in}\; \mathcal {T}\;\text{for}\; i\in\Phi \},\\
\widetilde{\E^{\Phi}(M,Y)}=\{t=(t_i)_{i\in\Phi}\in \E^{\Phi}(M,Y)\mid t_i \;\text{factors through}\; \Sigma^i(g) \;\text{in}\; \mathcal {T}\;\text{for}\; i\in\Phi \},\\
\widetilde{\E^{\Phi}(Y)}=\{t=(t_i)_{i\in\Phi}\in \E^{\Phi}(Y)\mid g t_i \;\text{factors through}\; \Sigma^i(g)\;\text{in}\; \mathcal {T}\;\text{for}\; i\in\Phi \}.
\end{array}
$$
This construction originally comes from Chen's construction in an abelian category in \cite{Chen2014} and Pan's construcion in  an additive category in \cite{Pan}.

\begin{lem}\cite[Lemma\; 4.2]{Xi}\label{lem3}
Suppose that $\Lambda$ is a subring of $\Gamma$ with the same identity.
\begin{itemize}

\item[{\rm (1)}] The restriction functor $\mathcal {F}: \Gamma\modc\ra \Lambda\modc$ is an exact faithful functor, and has a right adjoint
$\mathcal {G}=\Hom(_\Lambda\Gamma_\Gamma,-): \Lambda\modc\ra \Gamma\modc$ and a left adjoint $\mathcal {H}=\Gamma\otimes_\Lambda-: \Lambda\modc\ra \Gamma\modc$. In particular, $\mathcal {H}$ preserves projective modules and
$\mathcal {G}$ preserves injective modules.
\item[{\rm (2)}] The functor $\mathcal {H}=\Gamma\otimes_\Lambda-: \proj{\Lambda}\ra \proj{\Gamma}$ which sends $\Lambda$ to $\Gamma$ is faithful.
\end{itemize}
\end{lem}

Recall that an additive category $\mathcal{C}$ is idempotent complete (or Karoubi envelope) if for every
idempotent $p: C\ra C$, that is, $p^2 = p$, there is a decomposition $C\simeq K\oplus K'$ such that $p\simeq \begin{pmatrix}
0&0 \\
0&1
\end{pmatrix}$.
Note that the additive category $\mathcal{C}$ is idempotent complete if and only if every idempotent has a kernel.
For more details of Karoubi's construction on the idempotent complete of an additive category, we refer to \cite{Ka}.
Let $\mathcal{C}$ be an additive category. The idempotent completion
of $\mathcal{C}$ is denoted by $\widehat{\mathcal{C}}$ and is defined as follows. The objects of $\widehat{\mathcal{C}}$ are the pairs $(C, p)$, where $C$ is an object of $\mathcal{C}$ and $p:C\ra C$ is an idempotent morphism. A morphism in $\widehat{\mathcal{C}}$ from $(C, p)$
to $(D, q)$ is a morphism $f:C\ra D \in\mathcal{C}$ such that $fp=qf=f$. For any object $(C, p)$ in
$\widehat{\mathcal{C}}$, the identity morphism $1_{(C, p)} = p$.
There is a fully faithful additive functor $i_{\mathcal {C}}:\mathcal {C}\ra \widehat{\mathcal{C}}$ defined as follows. For
an object $C$ in $\mathcal{C}$, we have that $i_{\mathcal {C}}(C)=(C, 1_C)$ and for a morphism $f$ in $\mathcal {C}$, we have that
$i_{\mathcal {C}}(f)=f$. Every additive category $\mathcal{C}$ can be fully faithfully embedded into an idempotent complete additive category $\widehat{\mathcal{C}}$.

\begin{lem}\cite{Pan}
Let $\al C$ be an additive category.
Then the additive functor $\Hom_{\widehat{\al C}}((W,1_W),-): \widehat{\add W}\lra \Gamma\proj$ is an equivalence of additive categories, where $\widehat{\add W}$ is the idempotent complete of $\add W$.
\end{lem}

Let $\mathcal{G}$ be the inverse functor of $\Hom_{\widehat{\al C}}((W,1_W),-)$, and denote by $\mathcal{F}$ the composition of the functors  $-_\Lambda\otimes\Gamma$ and $\mathcal{G}$.
We then have the following commutative diagram:

$$\xymatrix{
\proj{\Lambda} \ar[rrr]^{\Gamma\otimes_\Lambda-}\ar[drrr]_{\mathcal{F}}&&& \proj{\Gamma}\ar@<1ex>[d]^{\mathcal{G}}\\
&&&\widehat{\add W}\ar@<1ex>[u]^{\Hom_{\mathcal {C}}(W,-)}
.}$$
By Lemma \ref{lem3}, we know that $\mathcal{F}$ is a faithful functor. Denote the image of $F$ by $\mathcal {S}$. Then $\mathcal {S}$
is a subcategory of $\widehat{\add V}$, but it is not necessarily a full subcategory of $\widehat{\add V}$.

As in \cite{Pan}, we have the following isomorphisms:
$$\aligned
\widetilde{\E^{\Phi}(X,M)}\simeq\Hom_{\Lambda}(\Lambda e_{11},\Lambda e_{22})
\simeq\widetilde{\Hom_{\Gamma}(\Gamma e_{11},\Gamma e_{22})}
\simeq\Hom_{\mathcal {S}}(\mathcal{G}(\Gamma e_{11}),\mathcal{G}(\Gamma e_{22}))\\
\simeq\Hom_{\mathcal {S}}(X,M),
\endaligned$$
where $e_{11}, e_{22}$ are idempotents of $\Lambda$, and $\widetilde{\Hom_{\Gamma}(\Gamma e_{11},\Gamma e_{22})}$ is a subring of $\Hom_{\Gamma}(\Gamma e_{11},\Gamma e_{22})$.
Consequently,
$$\aligned
\widetilde{\E^{\Phi}(X,M)}\simeq\Hom_{\mathcal {S}}(X,M),
\widetilde{\E^{\Phi}(M,Y)}\simeq\Hom_{\mathcal {S}}(M,Y),
\widetilde{\E^{\Phi}(X)}\simeq\End_{\mathcal {S}}(X),\\
\E^{\Phi}(M)\simeq\End_{\mathcal {S}}(M), \E^{\Phi}(M,X)\simeq\Hom_{\mathcal {S}}(M,X), \E^{\Phi}(Y,M)\simeq\Hom_{\mathcal {S}}(Y,M),\\
\widetilde{\E^{\Phi}(M,Y)}\simeq\Hom_{\mathcal {S}}(M,Y), \E^{\Phi}(Y,X)\simeq\Hom_{\mathcal {S}}(Y,X), \widetilde{\E^{\Phi}(Y)}\simeq\End_{\mathcal {S}}(Y),\\
\endaligned$$
and
$$\aligned
\widetilde{\E^{\Phi}(X,M_i)}\simeq\Hom_{\mathcal {S}}(X,M_i),
\widetilde{\E^{\Phi}(M_i,Y)}\simeq\Hom_{\mathcal {S}}(M_i,Y),
\E^{\Phi}(M_i,X)\simeq\End_{\mathcal {S}}(M_i,X),\\
\E^{\Phi}(Y,M_i)\simeq\Hom_{\mathcal {S}}(Y,M_i),
\E^{\Phi}(M_i,M_j)\simeq\Hom_{\mathcal {S}}(M_i,M_j),
\endaligned$$
for $1\leq i,j\leq n$.

Thus transferring the subalgebras of  $\Phi$-Auslander-Yoneda algebras to endomorphism algebras of some sequence, one may ask whether this sequence
$X\xra f M_1\xra{f_1} M_2\ra\cdots\xra{f_{n-3}} M_{n-2}\xra g Y$   is an $\add M$-split sequence or not in $\mathcal {S}$.
The following is a key lemma in our proof which tells us that $f$ and $g$ are left and right $\add M$-approximations of $X$ and $Y$ in $\mathcal {S}$, respectively.

\begin{lem}\label{3.4} The morphisms $f$ and $g$ are left and right $\add M$-approximations of $X$ and $Y$ in $\mathcal {S}$, respectively. Moreover, there are two exact sequences induced by the definitions of subrings:
$$
0\lra \Hom_{\mathcal {S}}(M, X)\xra {\Hom_{\mathcal {S}}(M, f)} \Hom_{\mathcal {S}}(M,  M_1)\ra\cdots\ra \Hom_{\mathcal {S}}(M, M_{n-2})\xra {\Hom_{\mathcal {S}}(M, g)} \Hom_{\mathcal {S}}(M, Y)\lra 0,
\quad (\star),$$
and
$$0\lra \Hom_{\mathcal {S}}(Y, M)\xra {\Hom_{\mathcal {S}}(g, M)} \Hom_{\mathcal {S}}(M_{n-2}, M)\ra\cdots\ra \Hom_{\mathcal {S}}(M_1,M)\xra {\Hom_{\mathcal {S}}(f, M)} \Hom_{\mathcal {S}}(X, M)\lra 0\quad (\star\star).
$$
\end{lem}

\begin{proof}
Since $\widetilde{\E^{\Phi}(X,M)}\simeq\Hom_{\mathcal {S}}(X,M)$ and $\widetilde{\E^{\Phi}(M,Y)}\simeq\Hom_{\mathcal {S}}(M,Y)$,
we have the following exact sequences,
$\Hom_{\mathcal {S}}(M_1,M)\xra{\Hom_{\mathcal {S}}(f,M)} \Hom_{\mathcal {S}}(X,M)\ra 0,
\;\text{and}\;
\Hom_{\mathcal {S}}(M,M_n)\xra{\Hom_{\mathcal {S}}(M,g)} \Hom_{\mathcal {S}}(M,Y)\ra 0.$
Hence, the morphism $X\xra f M_1$ is a left $\add M$-approximation in $\mathcal {S}$, and the morphism $M_n\xra g Y$ is a right $\add M$-approximation in $\mathcal {S}$.
Since
$\Hom_{\al T}(M,\Sigma^i(X))=0=\Hom_{\al T}(Y,\Sigma^i(M))$ for $0\neq i\in\Phi$,
and $\Hom_{\al T}(\Sigma(X),Y)=\Hom_{\al T}(\Sigma(X),M)=\Hom_{\al T}(\Sigma(M),Y)=0$, there are two exact sequences
$$
0\lra \Hom_{\mathcal {S}}(M, X)\xra {\Hom_{\mathcal {S}}(M, f)} \Hom_{\mathcal {S}}(M,  M_1)\ra\cdots\ra \Hom_{\mathcal {S}}(M, M_{n-2})\xra {\Hom_{\mathcal {S}}(M, g)} \Hom_{\mathcal {S}}(M, Y)\lra 0,
$$
and
$$0\lra \Hom_{\mathcal {S}}(Y, M)\xra {\Hom_{\mathcal {S}}(g, M)} \Hom_{\mathcal {S}}(M_{n-2}, M)\ra\cdots\ra \Hom_{\mathcal {S}}(M_1,M)\xra {\Hom_{\mathcal {S}}(f, M)} \Hom_{\mathcal {S}}(X, M)\lra 0.
$$
\end{proof}

Set the complex $\cpx{P}:  0\lra X\xra {f}  M_1 \xra {f_1}\cdots\lra  M_{n-1}
\xra {\overline{f_{n-3}}} M_{n-2}\oplus M\lra 0$, where $\overline{f_{n-3}}=(0,f_{n-3}):M_{n-1}\ra M_{n-2}\oplus M$.
Since $H^i(\Hom_{\al S}(M,\cpx{P}))=0$ for all $i\neq n$ and $H^i(\Hom_{\al S}(\cpx{P},M))=0$ for all $i\neq -n$,
it follows from \cite[Lemma 2.1]{Hoshino2003} that the complex $\cpx{P}$
is self-orthonal in $\Kb{\mathcal{S}}$.

Let $U=M\oplus X$ and $V=M\oplus Y$ . Then we have
$$\End_{\mathcal{S}}(U)\simeq\begin{pmatrix}
\End_{\mathcal {S}}(M)&       \Hom_{\mathcal {S}}(M, X)\\
\Hom_{\mathcal {S}}(X,M)     & \End_{\mathcal {S}}(X)
\end{pmatrix}
=\begin{pmatrix}
\E^{\Phi}(M)&       \E^{\Phi}(M, X)\\
\widetilde{\E^{\Phi}(X,M})     & \widetilde{\E^{\Phi}(X)}
\end{pmatrix}=:A.
$$
and
$$\End_{\mathcal{S}}(V)\simeq\begin{pmatrix}
\End_{\mathcal {S}}(M)&       \Hom_{\mathcal {S}}(M, Y)\\
\Hom_{\mathcal {S}}(Y,M)     & \End_{\mathcal {S}}(Y)
\end{pmatrix}
=\begin{pmatrix}
\widetilde{\E_{\al T}^\Phi(Y)}&     \E_{\al T}^\Phi(Y,M)\\
\widetilde{\E_{\al T}^\Phi(M,Y)}    &\E_{\al T}^\Phi(M)
\end{pmatrix}=:B.
$$
Therefore, we have a complex over $\End_{\mathcal{S}}(U)$ of the form
$$
\cpx{T}: 0\lra \Hom_{\mathcal {S}}(U, X)\xra {\Hom_{\mathcal {S}}(U, f)} \Hom_{\mathcal {S}}(U, M_1)\ra\cdots\ra \Hom_{\mathcal {S}}(U, M_n\oplus M)\lra 0.
$$
Then we have to show that $\cpx{T}$ is a tilting complex over $\End_{\mathcal{S}}(U)$.
The complex $\cpx{T}$ is self-orthogonal since $\Hom_{\mathcal{S}}(U,-):\add U\ra \proj{\End_{\mathcal{S}}(U)}$ is fully faithful.
It is easy to see that $\add(\cpx{T})$ generates $\Kb{\proj{\End_{\mathcal{S}}(U)}}$ as a triangulated category.
It suffices to show that  $\End_{\Kb{\proj{\End_{\mathcal{S}}(U)}}}(\cpx{T})\simeq B$ as rings.

\begin{lem}\label{2.5} The two rings $\End_{\Kb{\proj{\End_{\mathcal{S}}(U)}}}(\cpx{T})$ and $B$ are isomorphic.
\end{lem}

\begin{proof}

Let $\overline{f_{n-3}}=(0,f_{n-3}):M_{n-3}\ra M_{n-2}\oplus M$ and $\overline{g}=\begin{pmatrix}
1 &0 \\
0 &g
\end{pmatrix}: M_{n-2}\oplus M\ra Y\oplus M$. Then it follows that $\overline{f_{n-3}}\overline{g}=0$ from $f_{n-3}g=0$.
 Since $\Hom_{\al S}(U,-):\add(U)\ra \proj{\End_{\mathcal{S}}(U)}$
is fully faithful, we have the following isomorphim of rings
$$
\End_{\Kb{\proj{A}}}(\cpx{T})\simeq \End_{\Kb{\al S}}(\cpx{P}),
$$
To show the claim, it suffices to prove that there is a ring isomorphism
$$
\Theta: \End_{\Kb{\al S}}(\cpx{P})\lra \End_{\al S}(M\oplus Y).
$$

Now let $(\alpha,\alpha_1,\cdots, \alpha_n):\cpx{P}\lra\cpx{P}$ be a chain map between $\cpx{P}$ with $\alpha\in\End_{\al S}(X), \alpha^i\in\End_{\al S}(M_i)$ for $i=1,\cdots, n-3$ and $\alpha^n\in\End_{\al S}(M\oplus M_{n-2})$. Since $\alpha\in\End_{\al S}(X)\simeq\E^{\Phi}(X)$ and $\Hom_{\mathcal {S}}(M_i,M_j)\simeq\E^{\Phi}(M_i,M_j)$, and the commutativity between the two $n$-angles, there exists a morphism $\beta_i\colon M\oplus Y\ra \Sigma^i(M\oplus Y)$ such that the following diagram is commutative:

$$\xymatrix@M=0.1mm{
X\ar[r]^{f}\ar[d]^{\alpha_i} &M_1\ar[r]^{f_1}\ar[d]^{\alpha^1_i} & M_2\ar[r]^{f_2}\ar[d]^{\alpha^2_i} &\cdots\ar[r]^{f_{n-2}}&M_{n-3}\ar[d]^{\alpha^{n-3}_i}\ar[r]^{\overline{f_{n-3}}} &M\oplus M_{n-2}\ar[d]^{\alpha^{n-2}_i}\ar[r]^{\overline{g}} &M\oplus Y\ar@{..>}[d]^{\beta_i}\ar[r]^{\overline{h}} &\Sigma(X)\ar[d]^{\Sigma(\alpha_i)}\\
\Sigma^i(X)\ar[r]_{\Sigma^i(f)} &\Sigma^i(M_1)\ar[r]_{\Sigma^i(f_1)} &\Sigma^i(M_2)\ar[r]_{\Sigma^i(f_2)} &\cdots\ar[r]_{\Sigma^i(f_{n-2})}& \Sigma^i(M_{n-3})\ar[r]_{\Sigma^i(f_{n-3})} \ar[r]_{\Sigma^i(\overline{f_{n-3}})} &\Sigma^i(M\oplus M_{n-2})\ar[r]_{\Sigma^i(\overline{g})} &\Sigma^i(M\oplus Y)\ar[r]_{\Sigma^i(\overline{h})} &\Sigma^{i+1}(X)
.}$$
Therefore, there are two things to check, one is that the morphism $\beta=(\beta_i)_{i\in\Phi}$ is in $\End_{\mathcal{S}}(M\oplus Y)$, the other is that the morphism $\beta$ is unique.
To show these things, we need the following facts:
$$\End_{\mathcal{S}}(M\oplus Y)\simeq\begin{pmatrix}
\End_{\mathcal {S}}(M)&       \Hom_{\mathcal {S}}(M, Y)\\
\Hom_{\mathcal {S}}(Y,M)     & \End_{\mathcal {S}}(Y)
\end{pmatrix}
\simeq\begin{pmatrix}
\E^{\Phi}(M)&      \widetilde{ \E^{\Phi}(M, Y)}\\
\E^{\Phi}(Y,M)     & \widetilde{\E^{\Phi}(Y)}
\end{pmatrix}
$$
$$
\aligned
\End_{\mathcal{S}}(M\oplus M_{n-2})\simeq\begin{pmatrix}
\End_{\mathcal {S}}(M)&       \Hom_{\mathcal {S}}(M, M_{n-2})\\
\Hom_{\mathcal {S}}(M_{n-2},M)     & \End_{\mathcal {S}}(M_{n-2})
\end{pmatrix}
\simeq\begin{pmatrix}
\E^{\Phi}(M)&       \E^{\Phi}(M, M_{n-2})\\
\E^{\Phi}(M_{n-2},M)     & \E^{\Phi}(M_{n-2})
\end{pmatrix}\\=\E^{\Phi}(M\oplus M_{n-2}).
\endaligned
$$

To check $\beta\in\End_{\al S}(Y\oplus M)$, we set
$
\alpha^{n-2}_i=\begin{pmatrix}
x_i^1&x_i^2 \\
x_i^3 &x_i^4
\end{pmatrix}\in
\begin{pmatrix}
\E^{\Phi}(M)&       \E^{\Phi}(M, M_{n-2})\\
\E^{\Phi}(M_{n-2},M)     & \E^{\Phi}(M_{n-2})
\end{pmatrix},
$
and
$
\beta_i=\begin{pmatrix}
\beta_i^1&\beta_i^2 \\
\beta_i^3 &\beta_i^4
\end{pmatrix}
$, where $\beta_i^1:M\ra\Sigma^i(M),\beta_i^2:M\ra\Sigma^i(Y),
\beta_i^3:Y\ra\Sigma^i(M),\beta_i^4:Y\ra\Sigma^i(Y)$.
It follows from $\overline{g}\beta_i=\alpha_i^{n-2}\Sigma^i(\overline{g})$ that
$
\begin{pmatrix}
1&0 \\
0 &g
\end{pmatrix}
\begin{pmatrix}
\beta_i^1&\beta_i^2 \\
\beta_i^3 &\beta_i^4
\end{pmatrix}=\begin{pmatrix}
x_i^1&x_i^2 \\
x_i^3 &x_i^4
\end{pmatrix}
\begin{pmatrix}
1&0 \\
0 &\Sigma^i(g)
\end{pmatrix}.
$
Consequently,  we get $\beta_i^1=x_i^1, \beta_i^2=x_i^2\Sigma^i(g), g\beta_i^3=x_i^3, g\beta_i^4=x_i^4\Sigma^i(g)$. It follows that $\beta\in\begin{pmatrix}
\E_{\al T}^\Phi(M)&   \widetilde{ \E_{\al T}^\Phi(M,Y)} \\
  \E_{\al T}^\Phi(Y,M) &\widetilde{\E_{\al T}^\Phi(Y)}
\end{pmatrix}\simeq
\End_{\al S}(M\oplus Y)$. Then we define $\Theta(\alpha,\alpha_1,\cdots, \alpha_n)=\beta$.

To show that $\Theta$ is well-defined, it suffices to prove that the chain map $(\alpha,\alpha^1,\cdots, \alpha^{n-2})$ is homotopic to the zero if and only if $\beta=0$.
If $(\alpha,\alpha^1,\cdots, \alpha^{n-2})$ is null-homotopic, then there exists $h_i^{n-2}: M_n\oplus M\ra\Sigma^i( M_{n-3})$ such that $\alpha^{n-2}_i=h_i^{n-2}\Sigma^i(\overline{f_{n-3}})$. In this case, we have
$$
\alpha^{n-2}_i\Sigma^i(\overline{g})=h_i^{n-2}\Sigma^i(\overline{f_{n-3}})\Sigma^i(\overline{g})=0,
$$
and therefore $\overline{g}\beta_i=0$, hence $\begin{pmatrix}
1&0 \\
0 &g
\end{pmatrix}
\begin{pmatrix}
\beta_i^1&\beta_i^2 \\
\beta_i^3 &\beta_i^4
\end{pmatrix}=\begin{pmatrix}
\beta_i^1&\beta_i^2 \\
g\beta_i^3 &g\beta_i^4
\end{pmatrix}=0$, we thus get $\beta_i^1=\beta_i^2=0$.
In the case $i=0$,  it follows from $\overline{g}\beta_0=0$ that $\beta_0=\ol{h}a$, where $a:\Sigma(X)\to Y\oplus M$.
But $a=0$ since $\Hom_{\al T}(\Sigma(X),Y)=\Hom_{\al T}(\Sigma(X),M)=0$. Then we have $\beta_0=0$. In the case $0\neq i\in\Phi$,
Since $\Hom_{\al T}(Y,\Sigma^i(M))$ for $0\neq i\in\Phi=0$, $\beta_i^3=0$. We only to show that $\beta_i^4=0$.
By $h^1_i=0: M_1\to \Sigma^i(X)$, we get $\alpha_i=0$.
Then $\beta_i\Sigma^i(\ol{h})=0$, and $\beta_i$ factors through  $y_i=\begin{pmatrix}
y_i^1&y_i^2 \\
y_i^3 &y_i^4
\end{pmatrix}:Y\oplus M\to \Sigma^i{(Y\oplus M)}$, that is $\beta_i=\begin{pmatrix}
\beta_i^1&\beta_i^2 \\
\beta_i^3 &\beta_i^4
\end{pmatrix}=\begin{pmatrix}
y_i^1&y_i^2 \\
y_i^3 &y_i^4
\end{pmatrix}\Sigma^i(\ol{g})$, where $y_i^1:Y\to \Sigma^i(M_{n-2}), y_i^2:Y\to \Sigma^i(M)$, and $y_i^1=y_i^2=0$ since $\Hom_{\al T}(Y,\Sigma^i(M))$ for $0\neq i\in\Phi=0$.
Therefore, $\beta_i^4=0$ for $0\neq i\in\Phi$. Then $\beta_i=0$ for $i\in\Phi$ , hence $\beta=0$.

Now suppose that $\beta=0$. Then $\overline{g}\beta_i=\alpha_n\Sigma^i(\overline{g})=0$.  It follows that there exists a morphism $h_i^{n-2}: M\oplus M_{n-2}\ra \Sigma^i(M_{n-3})$ such that $\alpha_i^{n-2}=h_i^{n-2}\Sigma^i(\overline{f_{n-3}})$. Since $(\alpha_i^{n-3}-\overline{f_{n-3}}h_i^{n-2})\Sigma^i(\overline{f_{n-3}})=\alpha_i^{n-3}\Sigma^i(\overline{f_{n-3}})-\overline{f_{n-3}}h_i^{n-2}\Sigma^i(\overline{f_{n-3}})
=\overline{f_{n-3}}\alpha_i^{n-2}-\overline{f_{n-3}}\alpha_i^{n-2}=0$, there exists a morphism $h_i^{n-3}: M_{n-3}\ra \Sigma^i(M_{n-4})$ such that $\alpha_i^{n-3}=h_i^{n-3}\Sigma^i(f_{n-4})+\overline{f_{n-3}}h_i^{n-2}$.
By the induction, we have the following diagram:

$$\xymatrix@M=1mm{
X\ar[r]^{f}\ar[d]^{\alpha_i} &M_1\ar@{..>}[dl]_{h^1_i}\ar[r]^{f_1}\ar[d]^{\alpha_i^1} & M_2\ar@{..>}[dl]_{h^2_i}\ar[r]^{f_2}\ar[d]^{\alpha_i^2} &\cdots\ar[r]^{f_{n-4}}&M_{n-3}\ar@{..>}[dl]_{h^{n-3}_i}\ar[d]^{\alpha_i^{n-3}}\ar[r]^{\overline{f_{n-3}}} &M_{n-2}\oplus M\ar[d]^{\alpha_i^{n-2}}\ar[r]^{\overline{g}}\ar@{..>}[dl]_{h^{n-2}_i} &Y\oplus M\ar@{..>}[d]^{\beta_i}\ar[r]^{\overline{h}} &\Sigma(X)\ar[d]^{\Sigma(\alpha_i)}\\
\Sigma^i(X)\ar[r]_{\Sigma^i(f)} &\Sigma^i(M_1)\ar[r]_{\Sigma^i(f_1)} &\Sigma^i(M_2)\ar[r]_{\Sigma^i(f_2)} &\cdots\ar[r]_{\Sigma^i(d_{n-4})}& \Sigma^i(M_{n-3})\ar[r]_{\Sigma^i(\ol{f_{n-3}})} &\Sigma^i(M_{n-2}\oplus M)\ar[r]_{\Sigma^i(\ol{g})} &\Sigma^i(Y\oplus M)\ar[r]^{\Sigma^i(\overline{h})} &\Sigma^{i+1}(X)\\
}$$
and we get $\alpha_i^j=h_i^j\Sigma^i(d_{j-1})+d_jh_i^{j+1}$ for $j=2,\cdots, n-2$, $\alpha_i^1=h_i^1\Sigma^i(f)+d_1h_i^2$ and
$\alpha=f h_i^1$.
Therefore, $(\alpha,\alpha^1,\cdots, \alpha^{n-2})$ is null-homotopic in $\Kb{\al S}$.

So, by the above argument,  the following map
$$
 \begin{alignedat}{3}
\Theta: \End_{\Kb{\al S}}(\cpx{P})\lra \End_{\al S}(M\oplus Y)\\
\overline{(\alpha,\alpha_1,\cdots, \alpha_n)} \mapsto \beta
 \end{alignedat}
$$
is well-defined.  This map is injective, it remains to show that  $\Img(\Theta)=\End_{\al S}(M\oplus Y)$.
For any
$$
\beta=\begin{pmatrix}
\beta_i^1&\beta_i^2 \\
\beta_i^3 &\beta_i^4
\end{pmatrix}\in \begin{pmatrix}
\E^{\Phi}(M)&       \widetilde{\E^{\Phi}(M, Y)}\\
\E^{\Phi}(Y,M)     & \widetilde{\E^{\Phi}(Y)}
\end{pmatrix},
$$
it follows that there exist $\overline{\beta}_i^2\in\Hom(M,\Sigma^i(M_{n-2}))$, $\overline{\beta}_i^4\in\Hom(M_{n-2},\Sigma^i(M_{n-2}))$ such that $\beta_2=\overline{\beta}_i^2\Sigma^i(g ),g\beta_i^4=\overline{\beta}_i^4 \Sigma^i(g)$.
Consequently,
$$
\begin{pmatrix}
1&0 \\
0 &g
\end{pmatrix}
\begin{pmatrix}
\beta_i^1&\beta_i^2 \\
\beta_i^3 &\beta_i^4
\end{pmatrix}=\begin{pmatrix}
\beta_i^1&\overline{\beta}_i^2g \\
g\beta_i^3 &g\beta_4
\end{pmatrix}=\begin{pmatrix}
\beta_i^1&\overline{\beta}_i^2 \\
g\beta_i^3 &\overline{\beta}_i^4
\end{pmatrix}
\begin{pmatrix}
1&0 \\
0 &\Sigma^i(g)
\end{pmatrix}.
$$
Set $\alpha_i^{n-2}=\begin{pmatrix}
\beta_i^1&\overline{\beta}_i^2 \\
g\beta_i^3 &\overline{\beta}_i^4
\end{pmatrix}$.
Since $\overline{f_{n-3}}\alpha_i^{n-2}\Sigma^i(\overline{g})=0$, by induction, there exist maps
$\alpha_i^{n-3}: M_{n-3}\ra \Sigma^i(M_{n-3})$ such that $\alpha_i^j\Sigma^i(\overline{f_j})=\overline{f_j}\alpha_i^{j+1}$ for $1\leq j\leq n-3$, and there exists a unique map $\alpha\in \End_{\al S}(X)$ such that we have the following diagram
$$\xymatrix@M=0.1mm{
X\ar[r]^{f}\ar@{..>}[d]^{\alpha_i} &M_1\ar[r]^{f_1}\ar@{..>}[d]^{\alpha_i^1} & M_2\ar[r]^{f_2}\ar@{..>}[d]^{\alpha_i^2} &\cdots\ar[r]^{d_{n-4}}&M_{n-3}\ar@{..>}[d]^{\alpha_i^{n-3}}\ar[r]^{\overline{f_{n-3}}} &M_{n-2}\oplus M\ar@{..>}[d]^{\alpha_i^{n-2}}\ar[r]^{\overline{g}} &Y\oplus M\ar[d]^{\beta_i}\ar[r]^{\overline{h}} &\Sigma(X)\ar[d]^{\Sigma(\alpha_i)}\\
\Sigma^i(X)\ar[r]_{\Sigma^i(f)} &\Sigma^i(M_1)\ar[r]_{\Sigma^i(f_1)} &\Sigma^i(M_2)\ar[r]_{\Sigma^i(f_2)} &\cdots\ar[r]_{\Sigma^i(f_{n-4})}& \Sigma^i(M_{n-3})\ar[r]_{\Sigma^i(\ol{f_{n-3}})} &\Sigma^i(M_{n-2}\oplus M)\ar[r]_{\Sigma^i(\overline{g})} &\Sigma^i(Y\oplus M)\ar[r]_{\Sigma^i(\overline{h})} &\Sigma^{i+1}(X).
}$$
This implies that $\overline{(\alpha,\alpha^1,\cdots, \alpha^{n-2})}\in \End_{\Kb{\al S}}(\cpx{P})$ and $\beta=\Theta(\overline{(\alpha,\alpha^1,\cdots, \alpha^{n-2})})$, so the map $\Theta$ is surjective. Note that
$$
\End_{\mathcal {S}}(M\oplus Y)\simeq\begin{pmatrix}
\End_{\mathcal {S}}(M)&       \Hom_{\mathcal {S}}(M,Y)\\
\Hom_{\mathcal {S}}(Y,M)     & \End_{\mathcal {S}}(Y)
\end{pmatrix}=B,
$$
 This completes the proof.
\end{proof}

\begin{rem}
If $n=3$, then we re-obtain the main result  \cite[Theorem 1.1]{BP}.
\end{rem}

Finally, we give an explicit example which satisfies all conditions in Theorem \ref{1.1}.

\noindent{\bf Example.}
It is well-known that $2$-representation finite algebras of type A are the Auslander
algebras of $1$-representation finite algebras. For the definition and the properties of
$n$-representation finite algebras, we refer the reader to \cite{IO}.
Consider a $2$-representation finite algebra $\Lambda$ of type A which is the
Auslander algebra of $A_5$ which is given by quiver

$$
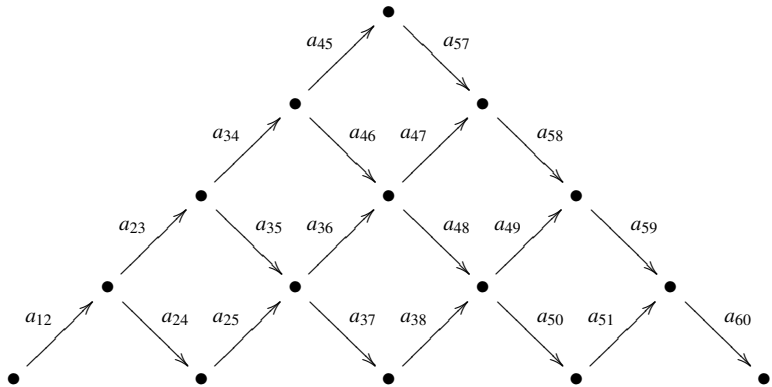
\xymatrix{
&&&&\bullet\\
&&&\bullet&& \bullet\\
&&\bullet&&\bullet&& \bullet\\
&\bullet&&\bullet&&\bullet&&\bullet\\
\bullet&&\bullet&&\bullet&&\bullet&&\bullet\\
\ar^{a_{12}}"5,1";"4,2"
\ar^{a_{23}}"4,2";"3,3"
\ar^{a_{34}}"3,3";"2,4"
\ar^{a_{45}}"2,4";"1,5"
\ar^{a_{24}}"4,2";"5,3"
\ar^{a_{35}}"3,3";"4,4"
\ar^{a_{46}}"2,4";"3,5"
\ar^{a_{57}}"1,5";"2,6"
\ar^{a_{25}}"5,3";"4,4"
\ar^{a_{36}}"4,4";"3,5"
\ar^{a_{47}}"3,5";"2,6"
\ar^{a_{58}}"2,6";"3,7"
\ar^{a_{37}}"4,4";"5,5"
\ar^{a_{48}}"3,5";"4,6"
\ar^{a_{59}}"3,7";"4,8"
\ar^{a_{38}}"5,5";"4,6"
\ar^{a_{49}}"4,6";"3,7"
\ar^{a_{50}}"4,6";"5,7"
\ar^{a_{51}}"5,7";"4,8"
\ar^{a_{60}}"4,8";"5,9"
}
$$
with relations $\{a_{23}a_{35}-a_{24}a_{25},a_{34}a_{46}-a_{35}a_{36},a_{45}a_{57}-a_{46}a_{47}, a_{36}a_{48}-a_{37}a_{38},a_{47}a_{58}-a_{48}a_{49},a_{49}a_{59}-a_{50}a_{51}\}$.

The $2$-cluster tilting subcategory $\mathcal {U}=\add\{\mathcal {S}_2^i( \Lambda)|i\in\mathbb{Z}\}$ is a $4$-angulated category with suspension functor $\Sigma_4$, where $\mathcal {S}=-\otimesL D \Lambda, \mathcal {S}_2=\mathcal {S}\circ[-2]$. And the Auslander-Reiten quiver of $\mathcal {U}$ is given as follows.

$$\widetilde{Q^{(2,5)}}:\quad\quad\begin{array}{l}
\vdots\\
 \xy (0,0)*+{\scriptstyle{400:0}}="a1",
(15,4)*+{\scriptstyle{310:0}}="a2",
(30,8)*+{\scriptstyle{220:0}}="a3",
(45,12)*+{\scriptstyle{130:0}}="a4",
(30,0)*+{\scriptstyle{301:0}}="a5",
(45,4)*+{\scriptstyle{211:0}}="a6",
(60,8)*+{\scriptstyle{121:0}}="a7",
(60,0)*+{\scriptstyle{202:0}}="a8",
(75,4)*+{\scriptstyle{112:0}}="a9",
(90,0)*+{\scriptstyle{103:0}}="a10",
{\ar^1 "a1"; "a2"},
{\ar "a2";"a3"},
{\ar "a3"; "a4"},
{\ar "a5"; "a6"},
 {\ar "a6"; "a7"},
 {\ar "a8"; "a9"},
 {\ar "a9"; "a10"},
 {\ar "a7"; "a9"},
 {\ar "a4"; "a7"},
{\ar "a6"; "a8"},
{\ar "a3"; "a6"},
{\ar "a2"; "a5"},
(15,24)*+{\scriptstyle{400:1}}="b1",
(30,28)*+{\scriptstyle{310:1}}="b2",
(45,32)*+{\scriptstyle{220:1}}="b3",
(60,36)*+{\scriptstyle{130:1}}="b4",
(45,24)*+{\scriptstyle{301:1}}="b5",
(60,28)*+{\scriptstyle{211:1}}="b6",
(75,32)*+{\scriptstyle{121:1}}="b7",
(75,24)*+{\scriptstyle{202:1}}="b8",
(90,28)*+{\scriptstyle{112:1}}="b9",
(105,24)*+{\scriptstyle{103:1}}="b10",
{\ar "b1"; "b2"},
{\ar "b2";"b3"},
{\ar "b3"; "b4"},
{\ar "b5"; "b6"},
{\ar "b6"; "b7"},
{\ar "b8"; "b9"},
{\ar "b9"; "b10"},
{\ar "b7"; "b9"},
{\ar "b4"; "b7"},
{\ar "b6"; "b8"},
{\ar "b3"; "b6"},
{\ar "b2"; "b5"},
{\ar "a5";"b1"},
{\ar "a6"; "b2"},
{\ar "a7"; "b3"},
{\ar "a8"; "b5"},
{\ar "a9"; "b6"},
{\ar "a10"; "b8"},
(30,48)*+{\scriptstyle{300:2}}="c1",
(45,52)*+{\scriptstyle{310:2}}="c2",
(60,56)*+{\scriptstyle{220:2}}="c3",
(75,60)*+{\scriptstyle{130:2}}="c4",
(60,48)*+{\scriptstyle{301:2}}="c5",
(75,52)*+{\scriptstyle{211:2}}="c6",
(90,56)*+{\scriptstyle{121:2}}="c7",
(90,48)*+{\scriptstyle{202:2}}="c8",
(105,52)*+{\scriptstyle{112:2}}="c9",
(120,48)*+{\scriptstyle{103:2}}="c10",
{\ar "c1"; "c2"},
{\ar "c2";"c3"},
{\ar "c3"; "c4"},
{\ar "c5"; "c6"},
{\ar "c6"; "c7"},
{\ar "c8"; "c9"},
{\ar "c9"; "c10"},
{\ar "c7"; "c9"},
{\ar "c4"; "c7"},
{\ar "c6"; "c8"},
{\ar "c3"; "c6"},
{\ar "c2"; "c5"},
{\ar "b5";"c1"},
{\ar "b6"; "c2"},
{\ar "b7"; "c3"},
{\ar "b8"; "c5"},
{\ar "b9"; "c6"},
{\ar "b10"; "c8"},
\endxy\\
\vdots
\end{array}$$

By \cite[Theorem 3.10]{IY}, we have an Auslander-Reiten $4$-angle
$$
211:0 \to 121:0\oplus 202:0 \to 220:1\oplus 301:1\oplus 112:0 \to 211:1\to \Sigma_4 (211:0).
$$
Set $X=211:0, Y=211:1$ and $M=121:0\oplus 202:0\oplus 220:1\oplus 301:1\oplus 112:0$.
Then the rings
$\begin{pmatrix}
\widetilde{\End(X)}&     \widetilde{\Hom(X,M)}  \\
\Hom(M,X)    &\End(M)
\end{pmatrix}$
and
$\begin{pmatrix}
\widetilde{\End(Y)}&     \Hom(Y,M)\\
\widetilde{\Hom(M,Y)}    &\End(M)
\end{pmatrix}$
are derived equivalent.\\

\noindent{\bf Acknowledgements.}
The authors are indebted to the referees for
reading the manuscript carefully and for giving valuable comments and suggestions.

{\footnotesize

\end{document}